\documentclass[a4paper,11pt]{article}
\usepackage[T1]{fontenc}
\usepackage[utf8]{inputenc}
\usepackage[english]{babel}
 \usepackage{lmodern}
\usepackage{soul}
\usepackage{eulervm}
\usepackage{amsmath,amsthm,amssymb,amsfonts}
\usepackage{mathtools}
\usepackage[normalem]{ulem}
\usepackage{xcolor}
\usepackage{enumerate}
\usepackage{graphicx}
\usepackage{geometry}
\usepackage{subcaption}

\usepackage{algorithm}
\usepackage{algorithmicx}
\usepackage{algpseudocode}
\usepackage[overload]{empheq}
\usepackage{framed}
\usepackage{stmaryrd}

\definecolor{rust}{rgb}{0.6,0.1,0.1}
\definecolor{mygray}{rgb}{0.9, 0.9, 0.9}

\usepackage{hyperref}
\hypersetup{
        colorlinks=true,   
        linkcolor=blue,    
        citecolor=rust,    
        urlcolor=blue      
}

\usepackage[framemethod=TikZ]{mdframed}
 \mdfsetup{
   roundcorner=8pt,
   backgroundcolor=gray!50!white!20,
   nobreak=true}

\newtheorem{theorem}{Theorem}[section]
\newtheorem{lemma}[theorem]{Lemma}
\newtheorem{proposition}[theorem]{Proposition}

\newtheorem{definition}[theorem]{Definition}

\theoremstyle{remark}
\newtheorem{remark}[theorem]{Remark}
\newtheorem{example}[theorem]{Example}

\newcommand{\la}{\lambda}
\newcommand{\bx}{\mathbf x}
\newcommand{\bw}{\mathbf w}
\newcommand{\by}{\mathbf y}
\newcommand{\bz}{\mathbf z}

\newcommand{\bu}{\mathbf u}

\newcommand{\bbz}{\bar{\mathbf z}}

\DeclareMathOperator{\dom}{dom}
\DeclareMathOperator{\gra}{gra}
\DeclareMathOperator{\zer}{zer}
\DeclareMathOperator{\Fix}{Fix}

\newcommand{\Hilbert}{\mathcal{H}}
\newcommand{\setto}{\rightrightarrows}
\DeclareMathOperator{\Id}{Id}

\newcommand{\integ}[2]{\llbracket{#1},{#2}\rrbracket}

\title{Distributed Forward-Backward Methods for Ring Networks}

\author{Francisco J.\ Arag\'on-Artacho\thanks{Department of Mathematics,
                             University of Alicante,
                             Alicante, \textsc{Spain}.
                                 Email:~\href{mailto:francisco.aragon@ua.es}
                                 {francisco.aragon@ua.es}}
        \and
       Yura Malitsky\thanks{Department of Mathematics,
                              Link\"oping University,
                              Linköping, \textsc{Sweden}.
                              Email:~\href{href:yurii.malitskyi@liu.se}
                                          {yurii.malitskyi@liu.se}}
         \and
        Matthew K. Tam\thanks{School of Mathematics \& Statistics,
                              University of Melbourne,
                              \textsc{Australia}.
                              Email:~\href{href:matthew.tam@unimelb.edu.au}
                                          {matthew.tam@unimelb.edu.au}}
         \and
        David Torregrosa-Bel\'en\thanks{Department of Mathematics,
                             University of Alicante,
                             Alicante, \textsc{Spain}.
                                 Email:~\href{mailto:david.torregrosa@ua.es}
                                 {david.torregrosa@ua.es}}
 }

\date{}

\graphicspath{{Figures/}}

\begin{document}

\maketitle

\begin{center}
\emph{Dedicated to the memory of Asen L.\ Dontchev}
\end{center}\medskip

\begin{abstract}
In this work, we propose and analyse forward-backward-type algorithms for
finding a zero of the sum of finitely many monotone operators, which are not
based on reduction to a two operator inclusion in the product space. Each
iteration of the studied algorithms requires one resolvent evaluation per
set-valued operator, one forward evaluation per cocoercive operator, and two
forward evaluations per monotone operator. Unlike existing methods, the
structure of the proposed algorithms are suitable for distributed, decentralised
implementation in ring networks without needing global summation to enforce consensus between nodes.
\end{abstract}


\section{Introduction}
In this work, we propose algorithms of forward-backward-type for solving structured monotone inclusions in a real Hilbert space $\Hilbert$. Specifically, we consider the problem
\begin{equation}\label{eq:mono inc}
 \text{find~}x\in\Hilbert\text{~such that~}0 \in \left(\sum_{i=1}^nA_i+\sum_{i=1}^mB_i\right)(x),
\end{equation}
where $A_1,\dots,A_n\colon\Hilbert\setto\Hilbert$ are maximally monotone
operators, and $B_1,\dots,B_m\colon\Hilbert\to\Hilbert$ are either cocoercive,
or monotone and Lipschitz continuous. Inclusions in the form \eqref{eq:mono inc}
arise in a number of settings of fundamental importance in mathematical
optimisation. In what follows, we describe three such examples.

\begin{example}[Composite minimisation] Consider the minimisation problem given by
\begin{equation}\label{eq:min}
 \min_{x\in\Hilbert}\sum_{i=1}^ng_i(x) + \sum_{i=1}^{m}f_i(x),
\end{equation}
where $g_1,\dots,g_n\colon\Hilbert\to(-\infty,+\infty]$ are proper, lsc and convex, and $f_1,\dots,f_{m}\colon\Hilbert\to(-\infty,+\infty)$ are convex and differentiable with $L$-Lipschitz continuous gradients. Through its first order optimality condition, \eqref{eq:min} can be posed as \eqref{eq:mono inc} with
$$A_i=\partial g_i\text{~~and~~}B_i=\nabla f_i$$
where $\partial g_i$ denotes the \emph{subdifferential} of $g_i$. Note that the operators $B_1,\dots,B_{m}$ are both $L$-Lipschitz and $\frac{1}{L}$-cocoercive, due to the \emph{Baillon--Haddad theorem}~\cite[Corolaire~10]{BaillonHaddad}.
\end{example}

\begin{example}[Structured saddle-point problems] Consider the saddle-point problem given by
\begin{equation}\label{eq:saddle}
\min_{x\in\Hilbert_1}\max_{y\in\Hilbert_2}\sum_{i=1}^nh_i(x)+\sum_{i=1}^{m}\Phi_i(x,y)-\sum_{i=1}^ng_i(y),
\end{equation}
where $h_1,\dots,h_n\colon\Hilbert_1\to(-\infty,+\infty]$, $g_1,\dots,g_n\colon\Hilbert_2\to(-\infty,+\infty]$ are proper, lsc and convex, and $\Phi_1,\dots,\Phi_{m}\colon\Hilbert_1\times\Hilbert_2\to(-\infty,+\infty]$ are differentiable convex-concave functions with Lipschitz continuous gradient.  Assuming a saddle-point exists, \eqref{eq:saddle} can be posed as \eqref{eq:mono inc} in the space $\Hilbert:=\Hilbert_1\times\Hilbert_2$ with
$$A_i(x,y)=\binom{\partial h_i(x)}{\partial g_i(y)}\text{~~and~~}B_i(x,y)=\binom{\phantom{-}\nabla_x\Phi_i(x,y)}{-\nabla_y\Phi_i(x,y)},$$
where we note that the operators $B_1,\dots,B_{n}\colon\Hilbert\to\Hilbert$ are monotone, due to \cite[Theorem~2]{rockafellar1970monotone}, and $L$-Lipschitz continuous, but generally not cocoercive.
\end{example}

\begin{example}[Structured variational inequalities]
Consider the variational inequality problem given by
\begin{equation}\label{eq:VI}
\text{find~}x^*\in\Hilbert\text{~such that~}\sum_{i=1}^ng_i(x)-\sum_{i=1}^ng_i(x^*)+\sum_{i=1}^{m}\langle B_i(x^*),x-x^*\rangle\geq 0\quad\forall x\in\Hilbert,
\end{equation}
where $g_1,\dots,g_n\colon\Hilbert\to(-\infty,+\infty]$ are proper, lsc and convex, and $B_1,\dots,B_{m}\colon\Hilbert\to\Hilbert$ are monotone and $L$-Lipschitz. Then \eqref{eq:VI} is of the form of \eqref{eq:mono inc} with $A_i=\partial g_i$. An important special case of \eqref{eq:VI} is the constrained variational inequality problem given by
$$ \text{find~}x^*\in\Hilbert\text{~such that~}\sum_{i=1}^{m}\langle B_i(x^*),x-x^*\rangle\geq 0\quad\forall x\in C:=\bigcap_{i=1}^nC_i, $$
where $C_1,\dots,C_n\subseteq\Hilbert$ are nonempty, closed and convex sets. This formulation allows one to exploit a representation of the set $C$ in terms of the simpler sets $C_1,\dots,C_n$.
\end{example}

\paragraph{Splitting algorithms} We focus on \emph{splitting
  algorithms} for solving \eqref{eq:mono inc} of forward-backward-type, by which
we mean those whose iteration can be expressed in terms of the \emph{resolvents}
of the set-valued operators $A_1,\dots,A_n$ and direct evaluations of the
single-valued operators $B_1,\dots,B_m$. It is always possible to reduce this problem to the $m=1$ case by combining the single-valued operators into a single operator $F:=\sum_{i=1}^mB_i$ whilst preserving the above features. However, since the resolvent of a sum is generally not related to the individual resolvents, the same cannot be said for the set-valued operators, and so it makes sense to distinguish algorithms for \eqref{eq:mono inc} based on the value of~$n$.

In the case $n=1$, there are many methods satisfying the above criteria. Among them, the best known are arguably the \emph{forward-backward method} given by
 $$ x^{k+1} = J_{\lambda A_1}\bigl(x^k-\lambda F(x^k)\bigr), $$
which can be used when $F$ is cocoercive, and the \emph{forward-backward-forward method}~\cite{tseng2000modified} given by
 $$\left\{\begin{aligned}
  y^{k} &= J_{\lambda A_1}\bigl(x^k-\lambda F(x^k)\bigr) \\
  x^{k+1} &= y^k-\lambda F(y^k)+\lambda F(x^k),
  \end{aligned}\right. $$
which can be used when $F$ is monotone and Lipschitz. When $n=2$, there are also many methods. For instance, if $F$ is cocoercive, \emph{Davis--Yin splitting}~\cite{davis2017three,dao2021adaptive,aragon2021direct} which takes the form
$$ \left\{\begin{aligned}
     x^k &= J_{\lambda A_1}(z^k) \\
     z^{k+1} &= z^k + J_{\lambda A_2}\bigl(2x^k-z^k-\lambda F(x^k)\bigr) - x^k
   \end{aligned}\right. $$
can be applied, and if $F$ is monotone and Lipschitz, then the \emph{backward-forward-reflected-backward methods}~\cite{rieger2020backward} can be used.

However, for $n>2$, the situation is drastically different. Most existing methods rely on a product space reformulation, either directly or implicitly.
For instance, the iteration given by
\begin{equation}\label{eq:gdy}
\left\{\begin{aligned}
 x^k   &= \frac{1}{n}\sum_{i=1}^nz_i^k \\
 z^{k+1}_i &= z_i^k + J_{\lambda A_i}\bigl(2x^k-z_i^k-\la B_i(x^k)\bigr)-x^k \qquad \forall i\in\integ{1}{n}
\end{aligned}\right.
\end{equation}
for cocoercive $B_1,\dots,B_n$, amounts to Davis--Yin splitting applied to the three operator inclusion
\begin{equation}\label{eq:product}
 \text{find~}\bx=(x,\dots,x)\in\Hilbert^n\text{~such that~}0\in (N_D+A+B)(\bx),
\end{equation}
where $A:=(A_1,\dots,A_n)$, $B:=(B_1,\dots,B_n)$ and $N_D$ denotes the \emph{normal cone} to the diagonal subspace $D:=\{(x_1,\dots,x_n)\in\Hilbert^n:x_1=\dots=x_n\}$. Other methods for \eqref{eq:mono inc} with $n>2$ include the \emph{generalised forward-backward method}~\cite{raguet2013generalized} and those from the \emph{projective splitting family}~\cite{johnstone2020projective,johnstone2021single}.

Indisputably, product space reformulations such as \eqref{eq:product} provide a convenient tool that makes the derivation of algorithms for $n>2$ operators an almost mechanical procedure. It is therefore natural to consider whether this tool is the only one at our disposal. In addition to academic importance in its own right, the discovery of new algorithms that do not fall within standard categories can provide new possibilities, both in terms of mathematical techniques and potential applications. Sometimes these applications can be quite unexpected, as we demonstrate next.

\paragraph{Distributed algorithms} Advances in hardware (parallel computation) and increasing the size of datasets (decentralised storage) have made distributed algorithms  one of the most prevalent trends in algorithm development. Such algorithms rely on a network of devices that perform  subtasks and  are able to communicate with each
other. For details on the topic, the reader is referred to the book of Bertsekas \& Tsitsiklis \cite{bertsekas1989parallel} as well as \cite{condat2020distributed} for recent advances.
 
 From the perspective of distributed computing, the product space formulation
 generally requires the computation of a global sum across all nodes in every iteration. To be more concrete, consider a distributed implementation of \eqref{eq:gdy} in which node $i$ performs the $z_i$-updates by using its operators, $A_i$ and $B_i$. To perform the $x$-update, the local variables $z_1,\dots,z_n$ must be aggregated and the result then broadcast to the entire network. There may be many reasons why this is not desirable including default network setting, privacy or cost issues.

 Another important aspect of distributed communication is parallelism and
 synchronisation. Return to our example involving \eqref{eq:gdy} from the
 previous paragraph, the product space reformulation  provides a fully parallel
 algorithm in the sense that all nodes performing $z$-update can compute their
 updates in parallel before sending to the central coordinator. This
 parallelisation comes at cost of requiring global synchronisation between
 nodes. Specially, the algorithm \eqref{eq:gdy} cannot move from \mbox{$k$-th} to \mbox{$(k+1)$-th} iteration until all nodes $1,\dots,n$ have completed their computation. This can be overcome with \emph{asynchronous algorithms}, that is, those which only require little or no global synchronisation. However, their development and mathematical analysis are significantly more delicate.

 \paragraph{Our contribution} We propose and analyse algorithms of
 forward-backward-type for solving \eqref{eq:mono inc} which exploit
 problem structure. Note that by using the zero operator in~\eqref{eq:mono inc} if necessary, we can always assume that $m=n-1$. Applied to this problem with cocoercive
 operators $B_1,\dots,B_{n-1}$, our algorithm can be expressed as the fixed point iteration
 $\bz^{k+1}=T(\bz^{k})$ based on the operator $T\colon\Hilbert^{n-1}\to\Hilbert^{n-1}$ given by
\begin{equation*}
T(\bz)
:=  \bz +
\gamma\begin{pmatrix}
x_2-x_1 \\
x_3-x_2 \\
\vdots \\
x_{n}-x_{n-1} \\
\end{pmatrix}
\end{equation*}
where $\bx=(x_1,\dots,x_n)\in\Hilbert^n$ depends on $\bz=(z_1,\dots,z_{n-1})\in\Hilbert^{n-1}$ and is given by
\begin{equation*}\left\{\begin{aligned}
    x_1&=J_{\lambda A_1}(z_1),  \\
    x_i&=J_{\lambda A_i}(z_i+x_{i-1}-z_{i-1}-\lambda B_{i-1}(x_{i-1}) ) \quad \forall i\in\integ{2}{n-1}, \\
    x_n&=J_{\lambda A_n}\bigl(x_1+x_{n-1}-z_{n-1}-\lambda B_{n-1}(x_{n-1})\bigr).
   \end{aligned}\right.
\end{equation*}
For the case where  $B_{i}$ are monotone and Lipschitz, the underlying operator is slightly more complicated and relies on an update similar to the one proposed in the \emph{forward-reflected-backward} method~\cite{malitsky2018forward}.

Overall, the notable characteristics of the algorithms we propose are:
\begin{itemize}
  \item They do not rely on existing product space reformulation: Instead, we
        extend the framework for backward operators, proposed in~\cite{malitsky2021resolvent}, which in turn is a generalisation of \cite{ryu2020uniqueness} for $n>3$.

  \item They are decentralised and can be naturally implemented on a ring network for communication.
  
  \item The order in which variables are updated can vary significantly between executions:  $z^{k+1}_{i}$ can be computed before evaluation of $z^{k}_{i+2},z^{k-1}_{i+3},\dots$. 
\end{itemize}
Importantly, we believe that our work is an important starting point towards a more general template that will allow for different network topologies.

\bigskip

The remainder of this work is structured as follows: In Section~\ref{s:prelim}, we recall notation and preliminaries for later use. In Section~\ref{s:dfb}, we introduce and analyse a forward-backward type algorithm for solving \eqref{eq:mono inc} with cocoercive operators. In Section~\ref{s:dforb}, we introduce and analyse a modification of the algorithm from Section~\ref{s:dfb} which can be used when $B_{1},\dots,B_m$ are not necessarily cocoercive.

\section{Preliminaries}\label{s:prelim}
Throughout this paper, $\Hilbert$ denotes a real Hilbert space
equipped with inner product $\langle\cdot , \cdot\rangle$ and induced
norm $\|\cdot\|$.
A \emph{set-valued operator} is a mapping $A:\Hilbert\setto\Hilbert$ that assigns to each point in $\Hilbert$ a subset of $\Hilbert$, i.e.,~$A(x)\subseteq \Hilbert$ for all $x\in \Hilbert$. In the case when $A$ always maps to singletons, i.e.,~$A(x)=\{u\}$ for all $x\in \Hilbert$, $A$ is said to be a \emph{single-valued mapping} and is denoted by $A:\Hilbert\to\Hilbert$. In an abuse of notation, we may write $A(x)=u$ when $A(x)=\{u\}$.  The \emph{domain}, the \emph{graph}, the set of \emph{fixed points} and the set of \emph{zeros} of $A$, are denoted, respectively, by $\dom A$, $\gra A$, $\Fix A$ and $\zer A$;~i.e.,
\begin{align*}
\dom A&:=\left\{x\in\Hilbert : A(x)\neq\varnothing\right\},&
\gra A&:=\left\{(x,u)\in\Hilbert\times\Hilbert : u\in A(x)\right\},\\
\Fix A&:=\left\{x\in\Hilbert : x\in A(x)\right\},&
\zer A&:=\left\{x\in\Hilbert : 0\in A(x)\right\}.
\end{align*}
The \emph{inverse operator} of
$A$, denoted by $A^{-1}$, is defined through
$x\in A^{-1}(u) \iff u\in A(x)$. The \emph{identity operator} is
denoted by $\Id$.
\begin{definition}\label{def:cocoercive}
An operator $B:\Hilbert \to \Hilbert$ is said to be
\begin{enumerate}[(i)]
\item\emph{$L$-Lipschitz continuous} for $L >0$ if
\begin{equation*}
\|B(x)-B(y)\| \leq L \|x-y\| \quad \forall x,y \in \Hilbert;
\end{equation*}
\item\emph{$\frac{1}{L}$-cocoercive} for $L >0$ if
\begin{equation*}
\langle B(x)-B(y), x-y \rangle \geq \frac{1}{L} \|B(x)- B(y)\|^2 \quad \forall x,y \in \Hilbert.
\end{equation*}
\end{enumerate}
\end{definition}
Note that, by the  Cauchy--Schwarz inequality, a $\frac{1}{L}$-cocoercive
operator is always $L$-Lipschitz continuous.

\begin{definition}\label{defT}
An operator $T:\Hilbert \to \Hilbert$ is said to be
\begin{enumerate}[(i)]
\item\label{defT:i}\emph{quasi-nonexpansive} if
\begin{equation*}
     \|T(x)-y\| \leq \|x-y\| \quad \forall x\in\Hilbert,\forall y\in\Fix T;
\end{equation*}
\item\label{defT:ii}\emph{nonexpansive} if it is $1$-Lipschitz continuous, i.e.,
\begin{equation*}
    \|T(x)-T(y)\| \leq \|x-y\| \quad \forall x,y \in \Hilbert;
\end{equation*}
\item\label{defT:iv}\emph{strongly quasi-nonexpansive} if there exists $\sigma>0$ such that
\begin{equation*}
\|T(x)-y\|^2 + \sigma \|(\Id-T)(x)\|^2 \leq \|x-y\|^2 \quad \forall x\in\Hilbert,\forall y\in\Fix T;
\end{equation*}
\item\label{defT:iii}\emph{averaged nonexpansive} if there exists $\alpha\in{(0,1)}$ such that
\begin{equation*}
\|T(x)-T(y)\|^2+\frac{1-\alpha}{\alpha} \|(\Id-T)(x)-(\Id-T)(y)\|^2 \leq \|x-y\|^2 \quad \forall x,y\in\Hilbert.
\end{equation*}
\end{enumerate}
In particular, the following implications hold: (\ref{defT:iii})$\Rightarrow$(\ref{defT:ii})$\Rightarrow$(\ref{defT:i})  and   (\ref{defT:iii})$\Rightarrow$(\ref{defT:iv})$\Rightarrow$(\ref{defT:i}).
\end{definition}
When we wish to explicitly specify the constants involved, we refer to the operators in Definition~\ref{defT}\eqref{defT:iv} and \eqref{defT:iii}, respectively, as  $\sigma$-strongly quasi-nonexpansive and $\alpha$-averaged nonexpansive. Since the mapping $\alpha\mapsto\frac{1-\alpha}{\alpha}$ is a bijection from $(0,1)$ to $(0,+\infty)$, there is a one-to-one relationship between the values of $\sigma$ in~\eqref{defT:iv} and $\alpha$ in \eqref{defT:iii}, with inverse relation given by $\sigma\mapsto\frac{1}{1+\sigma}$.
\begin{definition}
A set-valued operator $A:\Hilbert \setto \Hilbert$ is monotone if
\begin{equation*}
    \langle x-y,u-v\rangle \geq 0 \quad \forall (x,u),(y,v)\in\gra{A}.
\end{equation*}
Furthermore, $A$ is said to be maximally monotone if there exists no monotone operator $B:\Hilbert \setto \Hilbert$ such that $\gra{B}$ properly contains $\gra{A}$.
\end{definition}
\begin{proposition}[{\cite[Corollary~20.28]{bauschkecombettes}}]
Every continuous monotone operator with full domain is maximally monotone. In particular, every cocoercive operator is maximally monotone.
\end{proposition}

The resolvent operator, whose definition is given next, is one of the main building blocks of splitting algorithms.
\begin{definition}\label{def:resolvent}
Given an operator $A\colon\Hilbert\setto\Hilbert$, the \emph{resolvent} of $A$ with parameter $\gamma>0$ is the operator $J_{\gamma A}\colon\Hilbert\setto\Hilbert$ defined by $J_{\gamma A}:=(\Id+\gamma A)^{-1}$.
\end{definition}
\begin{proposition}[{\cite{Minty1962} or \cite[Corollary~23.11]{bauschkecombettes}}]
Let $A:\Hilbert\setto\Hilbert$ be monotone and let $\gamma>0$. Then
\begin{enumerate}[(i)]
\item $J_{\gamma A}$ is single-valued.
\item $\dom J_{\gamma A}=\Hilbert$ if and only if $A$ is maximally monotone.
\end{enumerate}
\end{proposition}

\section{A Distributed Forward-Backward Method}\label{s:dfb}
Let $n\geq 2$ and consider the problem
\begin{equation}\label{eq:mono incl}
  \text{find~}x\in\Hilbert\text{~such that~}0\in\left(\sum_{i=1}^nA_i+\sum_{i=1}^{n-1}B_i\right)(x),
\end{equation}
where $A_1,\dots,A_n\colon\Hilbert\setto\Hilbert$ are maximally
monotone and $B_1,\dots,B_{n-1}\colon\Hilbert\to\Hilbert$ are
$\frac{1}{L}$-cocoercive.

For the case when $B_1=\dots=B_{n-1}=0$,
Malitsky~\&~Tam~\cite{malitsky2021resolvent} proposed a
splitting algorithm with $(n-1)$-fold lifting for finding a zero of the
sum of $n\geq 2$ maximally monotone operators;  see also \cite{aragon2022primaldual} for recent extensions. In this section, we
adapt the methodology developed in~\cite{aragon2021direct} to obtain a
splitting method of forward-backward-type for the inclusion~\eqref{eq:mono incl} by modifying the splitting method in \cite{malitsky2021resolvent} without increasing the dimension of the ambient space.

\begin{mdframed}
Given $\lambda\in{(0,\frac{2}{L})}$ and
  $\gamma\in\bigl(0,1-\frac{\lambda L}{2}\bigr)$ and an initial point $\bz^0=(z_1^0,\dots,z_{n-1}^0)\in\Hilbert^{n-1}$, our proposed algorithm for~\eqref{eq:mono incl} generates two sequences, $(\bz^k)\subseteq\Hilbert^{n-1}$ and $(\bx^k)\subseteq\Hilbert^n$, according to
\begin{subequations}\label{eq:alg1}
\begin{equation}\label{eq:th T_A}
\bz^{k+1}
=  \bz^k +
\gamma\begin{pmatrix}
x_2^k-x_1^k \\
x_3^k-x_2^k \\
\vdots \\
x_{n}^k-x_{n-1}^k \\
\end{pmatrix}
\end{equation}
and
\begin{equation}\label{eq:th J_A}
\begin{cases}
    x_1^k=J_{\lambda A_1}(z_1^k),  \\
    x_i^k=J_{\lambda A_i}(z_i^k+x_{i-1}^k-z_{i-1}^k-\lambda B_{i-1}(x_{i-1}^k) ) \quad \forall i\in\integ{2}{n-1}, \\
    x_n^k=J_{\lambda A_n}\bigl(x_1^k+x_{n-1}^k-z_{n-1}^k-\lambda B_{n-1}(x_{n-1}^k)\bigr).
   \end{cases}
\end{equation}
\end{subequations}
\end{mdframed}

The structure of \eqref{eq:alg1} lends itself to a distributed
decentralised implementation, similar to the one in
\cite[Algorithm~2]{malitsky2021resolvent}. More precisely, consider
a cycle graph with $n$ nodes labeled $1$ through $n$. Each node in the
graph represents an agent, and two agents can communicate only if
their nodes are adjacent. In our setting, this means that Agent $i$ can only communicate with Agents $i-1$ and $i+1\mod{n}$, for $i\in\integ{1}{n}$. We assume
that each agent only knows its operators in~\eqref{eq:mono inc}. Specifically, we assume that only Agent $1$ knows the operator $A_1$ and that, for each $i\in\{2,\dots,n\}$, only Agent $i$ knows the operators $A_i$ and $B_{i-1}$. The responsibility of updating $x_i$ is assigned to Agent $i$ for all $i\in\{1,\dots,n\}$ and the responsibility of updating $z_i$ is assigned to Agent $i$ for $i\in\{2,\dots,n\}$.
Altogether, this gives rise to the protocol for distributed decentralised
implementation of \eqref{eq:alg1} described in
Algorithm~\ref{alg:DFB proto}.
\begin{algorithm}[!ht]
\caption{Protocol for distributed decentralised implementation of \eqref{eq:alg1}.}\label{alg:DFB proto}
\begin{algorithmic}[1]
\Require{Let $\lambda\in{(0,\frac{2}{L})}$ and
  $\gamma\in\bigl(0,1-\frac{\lambda L}{2}\bigr)$.}
\State{For each $i\in\llbracket 2,n\rrbracket $, Agent $i$ chooses $z_{i-1}^0\in\Hilbert$ and sends it to Agent $i-1$.}
\For{$k=0,1,\dots$}
  \State Agent $1$ computes $x_1^k=J_{\lambda A_1} (z_1^k)$ and sends it to Agents $2$ and $n$\;
  \For{$i = 2,\dots, n-1$}
  \State Agent $i$ computes
    \begin{equation*}\left\{\begin{aligned}
     x_i^k &=  J_{\lambda A_i} (z_i^k+x_{i-1}^k-z_{i-1}^k-\lambda B_{i-1}(x_{i-1}^k))\\
     z_{i-1}^{k+1} &= z_{i-1}^k + \gamma(x_i^k-x_{i-1}^k),
    \end{aligned}\right.\end{equation*}
   \hspace*{\algorithmicindent}\hspace*{\algorithmicindent}sends $x_i^k$ to Agent $i+1$ and sends
   $z_{i-1}^{k+1}$ to Agent $i-1$;
   \EndFor
   \State Agent $n$ computes
   \begin{equation*}\left\{\begin{aligned}
     x_n^k &=  J_{\lambda A_n} (x_1^k+x_{n-1}^k-z_{n-1}^k-\lambda B_{n-1}(x_{n-1}^k))\\
     z_{n-1}^{k+1} &= z_{n-1}^k + \gamma(x_n^k-x_{n-1}^k),
    \end{aligned}\right.\end{equation*}
    \hspace*{\algorithmicindent}sends $z_{n-1}^k$ to Agent $n-1$;
  \EndFor
\end{algorithmic}
\end{algorithm}

\begin{remark}[Termination criterion for Algorithm~\ref{alg:DFB proto}]
Let $(\mathbf{z})^k$ be the sequence generated by Algorithm~\ref{alg:DFB proto}. In order to detect termination, one could compute (possibly periodically) the residual given by 
\begin{equation*}
\|\mathbf{z}^{k+1}-\mathbf{z}^k\|^2 = \sum_{i=1}^{n-1}\|z_i^{k+1}-z_i^k\|^2.
\end{equation*}
The structure of this residual is suitable for the distributed implementation within the protocol in the algorithm. Indeed, the $ith$ term in the sum, given by $\|z_i^{k+1}-z_i^k\|^2$, can already be computed by Agent $i+1$, and therefore the full residual $\|\mathbf{z}^{k+1}-\mathbf{z}^k\|^2$ can be computed by a global summation and broadcast operation. The same stopping criterion can also be  applied to the algorithm presented in Section~\ref{s:dforb} generated by the iteration given in~\eqref{eq:th T_A lip}~and~\eqref{eq:th J_A lip}.
\end{remark}

In order to analyse convergence of \eqref{eq:alg1}, we introduce the underlying fixed point operator $T\colon\Hilbert^{n-1}\to\Hilbert^{n-1}$ given by
\begin{equation}\label{eq:T_A}
T(\bz)
:=  \bz +
\gamma\begin{pmatrix}
x_2-x_1 \\
x_3-x_2 \\
\vdots \\
x_{n}-x_{n-1} \\
\end{pmatrix}
\end{equation}
where $\bx=(x_1,\dots,x_n)\in\Hilbert^n$ depends on $\bz=(z_1,\dots,z_{n-1})\in\Hilbert^{n-1}$ and is given by
\begin{equation}\label{eq:xi}\left\{\begin{aligned}
    x_1&=J_{\lambda A_1}(z_1),  \\
    x_i&=J_{\lambda A_i}(z_i+x_{i-1}-z_{i-1}-\lambda B_{i-1}(x_{i-1}) ) \quad \forall i\in\integ{2}{n-1}, \\
    x_n&=J_{\lambda A_n}\bigl(x_1+x_{n-1}-z_{n-1}-\lambda B_{n-1}(x_{n-1})\bigr).
   \end{aligned}\right.
\end{equation}
In this way, the sequence $(\bz^k)$ given by \eqref{eq:th T_A} satisfies $\bz^{k+1}=T(\bz^k)$ for all $k\in\mathbb{N}$.

\begin{remark}
Note that, although the sum of cocoercive operators is cocoercive (see, e.g.,~\cite[Proposition~4.12]{bauschkecombettes}), considering the sum of $n-1$ operators in~\eqref{eq:mono inc} gives the freedom of either applying each operator as a forward step before the corresponding backward step, or to apply the sum of all of them before a particular backward step (by setting all the operators to be equal to zero except for one of them, which would be equal to the sum).
\end{remark}

\begin{remark}[Special cases]\label{r:special cases}
If $n=2$, then $x_1=x_{n-1}$ and $T$ in \eqref{eq:T_A} recovers the operator corresponding to \emph{Davis--Yin splitting} \cite{davis2017three,dao2021adaptive,aragon2021direct} for finding a zero of $A_1+A_2+B_1$. In turn, this includes the \emph{forward-backward algorithm} and \emph{Douglas--Rachford splitting} as special cases by further taking $A_1=0$ or $B_1=0$, respectively.

If $B_1=\dots=B_{n-1}=0$, then $T$ in \eqref{eq:T_A} reduces to the resolvent splitting algorithms proposed by the authors in \cite{malitsky2021resolvent}. This has been further  studied in \cite{bauschke2021splitting} for the particular case in which the operators $A_i$ are normal cones of closed linear subspaces.

Although the number of set-valued and single-valued monotone operators in \eqref{eq:mono incl} differ by one, it is straightforward to derive a scheme where this is not the case by setting $A_1=0$. In this case, $x_1=J_{\lambda A_1}(z_1)=z_1$ can be used to eliminate $x_1$ so that~\eqref{eq:T_A} and~\eqref{eq:xi} respectively become
\begin{equation*}
T(\bz)
:=  \bz +
\gamma\begin{pmatrix}
x_2-z_1 \\
x_3-x_2 \\
\vdots \\
x_{n}-x_{n-1} \\
\end{pmatrix}
\end{equation*}
where
\begin{equation*}
\left\{\begin{aligned}
    x_2 &=J_{\lambda A_2}(z_2-\lambda B_{1}(z_{1}) ), \\
    x_i&=J_{\lambda A_i}(z_i+x_{i-1}-z_{i-1}-\lambda B_{i-1}(x_{i-1}) ) \quad \forall i\in\integ{3}{n-1}, \\
    x_n&=J_{\lambda A_n}\bigl(z_1+x_{n-1}-z_{n-1}-\lambda B_{n-1}(x_{n-1})\bigr).
   \end{aligned}\right.
\end{equation*}
While at first it may seem unusual that the number of set-valued and single-valued monotone operators in~\eqref{eq:mono incl} are not the same, we note that this same situation arises in Davis--Yin splitting as described above.
\end{remark}

\begin{remark}
  The algorithm given by \eqref{eq:alg1} appears to be new even in the special case with $A_i = 0$ and $B_i=\nabla f_i$ for convex smooth functions $f_i$. In
  this case, one of the most popular algorithms for solving
  $\min_x\sum_if_i(x)$ in a decentralised way is EXTRA, proposed in \cite{shi2015extra}. They are similar in spirit, but also have quite
  different properties. In particular, the main update of EXTRA is
  \[\bx^{k+1} = (\Id + W)\bx^{k} - \widetilde W\bx^{k-1} - \la [\nabla f(\bx^{k})-\nabla f(\bx^{k-1})], \]
  where $W$ and $\widetilde W$ are certain mixing matrices and
  $\bx^{1} = W\bx^{0}-\la \nabla f(\bx^{0})$.
  Undoubtedly, an advantage of EXTRA is the
  ability to use a wider range of mixing matrices which, in terms of communication,
  generalises better for  network topology. 
\end{remark}

In what follows, we first describe the relationship between the solutions of the monotone inclusion~\eqref{eq:mono incl} and the fixed point set of the operator $T$ in \eqref{eq:T_A}.

\begin{lemma}\label{l:fixed points}
Let $n\geq 2$ and $\gamma,\lambda>0$. The following assertions hold.
\begin{enumerate}[(a)]
\item\label{l:fixed points b} If $\bar{x}\in \zer\left(\sum_{i=1}^nA_i+\sum_{i=1}^{n-1}B_i\right)$, then there exists $\bar{\bz}\in\Fix T$.
\item\label{l:fixed points c} If $(\bar{z}_1,\ldots\bar{z}_{n-1})\in\Fix T$, then $\bar{x}:=J_{\lambda A_{1}}(\bar{z}_1)\in \zer\left(\sum_{i=1}^nA_i+\sum_{i=1}^{n-1}B_i\right)$. Moreover,
\begin{equation}\label{eq:fp x}
 \bar{x}=J_{\lambda A_i}(\bar{z}_i-\bar{z}_{i-1}+\bar{x}-\lambda B_{i-1}(\bar{x}))=J_{\lambda A_n}(2\bar{x}-\bar{z}_{n-1}-\lambda B_{n-1}(\bar{x})),
\end{equation}
for all $i\in\integ{2}{n-1}$.
\end{enumerate}
Consequently,
$$ \Fix T\neq\varnothing \iff \zer\left(\sum_{i=1}^nA_i+\sum_{i=1}^{n-1}B_i\right)\neq\varnothing.$$
\end{lemma}
\begin{proof}
\eqref{l:fixed points b}:~Let $\bar{x}\in \zer\left(\sum_{i=1}^nA_i+\sum_{i=1}^{n-1}B_i\right)$. Then there exists $\mathbf{v}=(v_1,\dots,v_n)\in\Hilbert^n$ such that $v_i\in A_i(\bar{x})$ and $\sum_{i=1}^nv_i+\sum_{i=1}^{n-1}B_i(\bar{x})=0$. Define the vector $\bar{\bz}=(\bar{z}_1,\dots,\bar{z}_{n-1})\in\Hilbert^{n-1}$ according to
$$ \left\{\begin{aligned}
\bar{z}_1 &:= \bar{x}+\lambda v_1 \in (\Id+\lambda A_1)\bar{x}, \\
\bar{z}_i &:= \lambda v_i+\bar{z}_{i-1} +\lambda B_{i-1}(\bar{x})  \in (\Id+\lambda A_i)(\bar{x}) - \bar{x}+\bar{z}_{i-1}+\lambda B_{i-1}(\bar{x}),
   \end{aligned}\right.$$
for $i\in\integ{2}{n-1}$.
Then $\bar{x}=J_{\lambda A_1}(z_1)$ and $\bar{x}=J_{\lambda A_i}(\bar{z}_i-\bar{z}_{i-1}+\bar{x}-\lambda B_{i-1}(\bar{x}))$ for $i\in\integ{2}{n-1}$. Furthermore, we have
\begin{align*}
 (\Id+\lambda A_n)(\bar{x})\ni \bar{x}+\lambda v_n
 &= \bar{x}-\lambda v_1-\lambda\sum_{i=2}^{n-1}\bigl(v_i+B_{i-1}(\bar{x})\bigr)-\lambda B_{n-1}(\bar{x}) \\
 &= \bar{x}-(z_1-\bar{x})-\sum_{i=2}^{n-1}\bigl(\bar{z}_i-\bar{z}_{i-1}\bigr)-\lambda B_{n-1}(\bar{x}) \\
 &= 2\bar{x}-\bar{z}_{n-1}-\lambda B_{n-1}(\bar{x}),
\end{align*}
which implies that $\bar{x}=J_{\lambda A_n}(2\bar{x}-\bar{z}_{n-1}-\lambda B_{n-1}(\bar{x}))$. Altogether, it follows that $\bar{\bz}\in\Fix T$.

\eqref{l:fixed points c}:~Let  $\bar{\bz}\in\Fix T$ and set $\bar{x}:=J_{\lambda A_1}(\bar{z}_1)$. Then \eqref{eq:fp x} holds thanks to the definition of $T$. The definition of the resolvent therefore implies
$$ \left\{\begin{aligned}
    \lambda A_1(\bar{x}) &\ni \bar{z}_1-\bar{x}, \\
    \lambda A_i(\bar{x}) &\ni \bar{z}_i-\bar{z}_{i-1}-\lambda B_{i-1}(\bar{x}) \quad \forall i\in\integ{2}{n-1},  \\
    \lambda A_n(\bar{x}) &\ni \bar{x}-z_{n-1}-\lambda B_{n-1}(\bar{x}).
   \end{aligned}\right. $$
Summing together the above inclusions gives $\bar{x}\in \zer\left(\sum_{i=1}^nA_i+\sum_{i=1}^{n-1}B_i\right)$, as claimed.
\end{proof}

Next, we study the nonexpansivity properties of the operator $T$ in \eqref{eq:T_A}.

\begin{lemma}\label{l:T_A ne}
For all $\bz=(z_1,\dots,z_n)\in\Hilbert^{n-1}$ and $\bar{\bz}=(\bar{z}_1,\dots,\bar{z}_n)\in\Hilbert^{n-1}$, we have
\begin{multline}\label{eq:T_A ne}
 \|T(\bz)-T(\bar{\bz})\|^2 + \left(\frac{1-\gamma}{\gamma}-\frac{\lambda L}{2\gamma}\right)\|(\Id-T)(\bz)-(\Id-T)(\bar{\bz})\|^2 \\
 + \frac{1}{\gamma}\bigl\|\sum_{i=1}^{n-1}(\Id-T)(\bz)_i-\sum_{i=1}^{n-1}(\Id-T)(\bar{\bz})_i\bigr\|^2 \leq \|\bz-\bar{\bz}\|^2.
\end{multline}
In particular, if $\lambda\in\bigl(0,\frac{2}{L}\bigr)$ and $\gamma\in\bigl(0,1-\frac{\lambda L}{2}\bigr)$, then $T$ is $\alpha$-averaged for $\alpha=\frac{2\gamma}{2-\lambda L}\in(0,1)$.
\end{lemma}%
\begin{proof}
This proof mainly uses the monotonicity property of the operators $A_1,\ldots,A_n$ together with the cocoercivity property of the operators $B_1,\ldots,B_{n-1}$ to obtain some bounds which yield~\eqref{eq:T_A ne}, from where the averagedness of operator $T$ can be directly deduced. For convenience, denote $\bz^+:=T(\bz)$ and $\bar{\bz}^+:=T(\bar{\bz})$. Further, let $\bx=(x_1,\dots,x_n)\in\Hilbert^n$ be given by \eqref{eq:xi} and let $\bar{\bx}=(\bar{x}_1,\dots,\bar{x}_n)\in\Hilbert^n$ be given analogously. Since $z_1-x_1\in \lambda A_1(x_1)$ and $\bar{z}_1-\bar{x}_1\in \lambda A_1(\bar{x}_1)$, monotonicity of $\lambda A_1$ implies
\begin{equation}\label{eq:mono A1}
\begin{aligned}
 0 &\leq \langle x_1-\bar{x}_1,(z_1-x_1)-(\bar{z}_1-\bar{x}_1)\rangle \\
   &= \langle x_2-\bar{x}_1,(z_1-x_1)-(\bar{z}_1-\bar{x}_1)\rangle + \langle x_1-x_2,(z_1-x_1)-(\bar{z}_1-\bar{x}_1)\rangle.
\end{aligned}
\end{equation}
For $i\in\integ{2}{n-1}$, $z_i-z_{i-1}+x_{i-1}-x_i-\lambda B_{i-1}(x_{i-1}) \in \lambda A_i(x_i)$ and $\bar{z}_i-\bar{z}_{i-1}+\bar{x}_{i-1}-\bar{x}_i-\lambda B_{i-1}(\bar{x}_{i-1}) \in \lambda A_i(\bar{x}_i)$. Thus, monotonicity of $\lambda A_i$ yields
{\begin{align*}
 0 &\leq \langle x_i-\bar{x}_i,z_i-z_{i-1}+x_{i-1}-x_i-\lambda B_{i-1}(x_{i-1})\rangle\\
 &\quad-\langle x_i-\bar{x}_i,\bar{z}_i-\bar{z}_{i-1}+\bar{x}_{i-1}-\bar{x}_i-\lambda B_{i-1}(\bar{x}_{i-1})\rangle \\
 &= \langle x_i-\bar{x}_i,(z_i-z_{i-1}+x_{i-1}-x_i)-(\bar{z}_i-\bar{z}_{i-1}+\bar{x}_{i-1}-\bar{x}_i)\rangle \\
 &\quad-\lambda\langle x_i-\bar{x}_i,B_{i-1}(x_{i-1})-B_{i-1}(\bar{x}_{i-1})\rangle \\
 &= \langle x_{i+1}-\bar{x}_i,(z_i-x_i)-(\bar{z}_i-\bar{x}_i)\rangle + \langle x_i-x_{i+1},(z_i-x_i)-(\bar{z}_i-\bar{x}_i)\rangle \\
    &\quad- \langle x_i-\bar{x}_{i-1},(z_{i-1}-x_{i-1})-(\bar{z}_{i-1}-\bar{x}_{i-1})\rangle\\
    &\quad-\langle \bar{x}_{i-1}-\bar{x}_i,(z_{i-1}-x_{i-1})-(\bar{z}_{i-1}-\bar{x}_{i-1})\rangle \\
    &\quad-\lambda\langle x_i-\bar{x}_i,B_{i-1}(x_{i-1})-B_{i-1}(\bar{x}_{i-1})\rangle.
\end{align*}}
Summing this inequality for $i\in\integ{2}{n-1}$ and simplifying gives
\begin{equation}\label{eq:mono Ai}
\begin{aligned}
0 \leq  &\langle x_{n}-\bar{x}_{n-1},(z_{n-1}-x_{n-1})-(\bar{z}_{n-1}-\bar{x}_{n-1})\rangle \\
 &- \langle x_2-\bar{x}_{1},(z_{1}-x_{1})-(\bar{z}_{1}-\bar{x}_{1})\rangle
   +\sum_{i=2}^{n-1}\langle x_i-x_{i+1},(z_i-x_i)-(\bar{z}_i-\bar{x}_i)\rangle\\
   &-\sum_{i=1}^{n-2}\langle \bar{x}_{i}-\bar{x}_{i+1},(z_{i}-x_{i})-(\bar{z}_{i}-\bar{x}_{i})\rangle \\
   &- \lambda\sum_{i=2}^{n-1}\langle x_i-\bar{x}_i,B_{i-1}(x_{i-1})-B_{i-1}(\bar{x}_{i-1})\rangle.
\end{aligned}
\end{equation}
Since $x_1+x_{n-1}-x_n-z_{n-1}-\lambda B_{n-1}(x_{n-1})\in \lambda A_n(x_n)$ and $\bar{x}_1+\bar{x}_{n-1}-\bar{x}_n-\bar{z}_{n-1}-\lambda B_{n-1}(\bar{x}_{n-1})\in \lambda A_n(\bar{x}_n)$, monotonicity of $\lambda A_n$ gives
{\begin{equation}\label{eq:mono An}
\begin{aligned}
 0 &\leq \langle x_n-\bar{x}_n,x_1+x_{n-1}-x_n-z_{n-1}-\lambda B_{n-1}(x_{n-1})\rangle\\
 &\quad -\langle x_n-\bar{x}_n,\bar{x}_1+\bar{x}_{n-1}-\bar{x}_n-\bar{z}_{n-1}-\lambda B_{n-1}(\bar{x}_{n-1})\rangle  \\
  &= \langle x_n-\bar{x}_n,(x_1-x_n)-(\bar{x}_1-\bar{x}_n)\rangle \\
  &\quad+ \langle x_n-\bar{x}_n,(x_{n-1}-z_{n-1})-(\bar{x}_{n-1}-\bar{z}_{n-1})\rangle \\
  &\quad - \lambda\langle x_n-\bar{x}_n,B_{n-1}(x_{n-1})-B_{n-1}(\bar{x}_{n-1})\rangle \\
   &= -\langle x_n-\bar{x}_{n-1},(z_{n-1}-x_{n-1})-(\bar{z}_{n-1}-\bar{x}_{n-1})\rangle\\
   &\quad + \langle \bar{x}_n-\bar{x}_{n-1},(z_{n-1}-x_{n-1})-(\bar{z}_{n-1}-\bar{x}_{n-1})\rangle \\
   &\quad  + \frac{1}{2}\left(\|x_1-\bar{x}_1\|^2-\|x_n-\bar{x}_n\|^2-\|(x_1-x_n)-(\bar{x}_1-\bar{x}_n)\|^2\right)  \\
   &\quad - \lambda\langle x_n-\bar{x}_n,B_{n-1}(x_{n-1})-B_{n-1}(\bar{x}_{n-1})\rangle.
\end{aligned}
\end{equation}}
Adding \eqref{eq:mono A1}, \eqref{eq:mono Ai} and \eqref{eq:mono An} and rearranging gives
\begin{equation}\label{eq:before key}
\begin{aligned}
0\leq &\sum_{i=1}^{n-1}\langle (x_i-\bar{x}_i)-(x_{i+1}-\bar{x}_{i+1}),\bar{x}_i-x_i\rangle\\
&+\sum_{i=1}^{n-1}\langle (x_i-\bar{x}_i)-(x_{i+1}-\bar{x}_{i+1}),z_i-\bar{z}_i\rangle \\
&+ \frac{1}{2}\left(\|x_1-\bar{x}_1\|^2-\|x_n-\bar{x}_n\|^2-\|(x_1-x_n)-(\bar{x}_1-\bar{x}_n)\|^2\right)\\
&-\lambda\sum_{i=1}^{n-1}\langle x_{i+1}-\bar{x}_{i+1},B_{i}(x_{i})-B_{i}(\bar{x}_{i})\rangle.
\end{aligned}
\end{equation}
The first term in \eqref{eq:before key} can be expressed as
\begin{equation}\label{eq:ip x}
\begin{aligned}
&\sum_{i=1}^{n-1}\langle (x_i-\bar{x}_i)-(x_{i+1}-\bar{x}_{i+1}),\bar{x}_i-x_i\rangle\\
&\quad= \frac{1}{2}\sum_{i=1}^{n-1}\left(\|x_{i+1}-\bar{x}_{i+1}\|^2-\|x_i-\bar{x}_i\|^2-\|(x_i-x_{i+1})-(\bar{x}_i-\bar{x}_{i+1})\|^2 \right) \\
&\quad= \frac{1}{2}\left(\|x_n-\bar{x}_n\|^2-\|x_1-\bar{x}_1\|^2 - \frac{1}{\gamma^2}\|(\bz-\bz^+)-(\bar{\bz}-\bar{\bz}^+)\|^2\right),
\end{aligned}
\end{equation}
and the second term in \eqref{eq:before key} can be written as
\begin{equation}\label{eq:ip xx}
\begin{aligned}
&\sum_{i=1}^{n-1}\langle (x_i-x_{i+1})-(\bar{x}_i-\bar{x}_{i+1}),z_i-\bar{z}_i\rangle \\
&\quad= \frac{1}{\gamma}\sum_{i=1}^{n-1}\langle (z_i-z_i^+)-(\bar{z}_i-\bar{z}_i^+),z_i-\bar{z}_i\rangle \\
&\quad= \frac{1}{\gamma}\langle (\bz-\bz^+)-(\bar{\bz}-\bar{\bz}^+),\bz-\bar{\bz}\rangle \\
&\quad= \frac{1}{2\gamma}\left(\|(\bz-\bz^+)-(\bar{\bz}-\bar{\bz}^+)\|^2+\|\bz-\bar{\bz}\|^2-\|\bz^+-\bar{\bz}^+\|^2 \right).
\end{aligned}
\end{equation}
To estimate the last term, Young's inequality and $\frac{1}{L}$-cocoercivity of $B_1,\dots,B_{n-1}$ gives
\begin{equation}\label{eq:ip coco}\begin{aligned}
-\sum_{i=1}^{n-1}&\langle x_{i+1}-\bar{x}_{i+1},B_i(x_{i})-B_i(\bar{x}_{i})\rangle \\
&= \sum_{i=1}^{n-1}\langle (\bar{x}_{i+1}-\bar{x}_{i})-(x_{i+1}-x_{i}),B_i(x_{i})-B_i(\bar{x}_{i})\rangle\\
&\quad + \sum_{i=1}^{n-1}\langle \bar{x}_{i}-x_{i},B_i(x_{i})-B_i(\bar{x}_{i})\rangle \\
&\leq \frac{L}{4} \sum_{i=1}^{n-1}\|(\bar{x}_{i+1}-\bar{x}_{i})-(x_{i+1}-x_{i})\|^2  + \frac{1}{L} \sum_{i=1}^{n-1}\|B_i(x_{i})-B_i(\bar{x}_{i})\|^2 \\
&\quad- \frac{1}{L} \sum_{i=1}^{n-1}\|B_i(x_{i})-B_i(\bar{x}_{i})\|^2 \\
& = \frac{L}{4} \sum_{i=1}^{n-1}\|(\bar{x}_{i+1}-\bar{x}_{i})-(x_{i+1}-x_{i})\|^2 \\
&= \frac{L}{4\gamma^2}\|(\bz-\bz^+)-(\bar{\bz}-\bar{\bz}^+)\|^2.
\end{aligned}\end{equation}%
Thus, substituting \eqref{eq:ip x} and \eqref{eq:ip xx} into \eqref{eq:before key}, using \eqref{eq:ip coco} and simplifying gives the claimed inequality~\eqref{eq:T_A ne}.
Finally, to show that \eqref{eq:T_A ne} implies  $T$ is $\alpha$-averaged with $\alpha:=\frac{2\gamma}{2-\lambda L}$, note that $\alpha\in(0,1)$ and satisfies $\frac{1-\alpha}{\alpha} = \frac{1-\gamma}{\gamma}-\frac{\lambda L}{2\gamma}$. This completes the proof.
\end{proof}

The following theorem is our main convergence result regarding the algorithm given by \eqref{eq:alg1}.

\begin{theorem}\label{th:main}
Let $n\geq 2$, let $A_1,\dots,A_n\colon\Hilbert\setto\Hilbert$ be maximally monotone and let $B_1,\dots,B_{n-1}\colon\Hilbert\to\Hilbert$ be $\frac{1}{L}$-cocoercive with $\zer\left(\sum_{i=1}^nA_i+\sum_{i=1}^{n-1}B_i\right)\neq\varnothing$. Further, let $\lambda\in\bigl(0,\frac{2}{L}\bigr)$ and $\gamma\in\bigl(0,1-\frac{\lambda L}{2}\bigr)$. Given $\bz^0\in\Hilbert^{n-1}$, let $(\bz^k)\subseteq\Hilbert^{n-1}$ and $(\bx^k)\subseteq\Hilbert^n$ be the sequences given by \eqref{eq:alg1}.
Then the following assertions hold.
\begin{enumerate}[(a)]
\item\label{th:main_a} The sequence $(\bz^k)$ converges weakly to a point $\bz\in\Fix T$.
\item\label{th:main_b} The sequence $(\bx^k)$ converges weakly to a point $(x,\dots,x)\in\Hilbert^n$ with $x\in\zer\left(\sum_{i=1}^nA_i+\sum_{i=1}^{n-1}B_i\right)$.
\item\label{th:main_c} The sequence $\bigl(B_i(x^k_{i})\bigr)$ converges strongly to $B_i(x)$ for all $i\in\integ{1}{n-1}$.
\end{enumerate}
\end{theorem}
\begin{proof}
\eqref{th:main_a}:~Since $\zer\left(\sum_{i=1}^nA_i+\sum_{i=1}^{n-1}B_i\right)\neq\varnothing$, Lemma~\ref{l:fixed points}\eqref{l:fixed points b} implies $\Fix T\neq\varnothing$. Since $\lambda\in\bigl(0,\frac{2}{L}\bigr)$ and $\gamma\in\bigl(0,1-\frac{\lambda L}{2}\bigr)$,  Lemma~\ref{l:T_A ne} implies $T$ is averaged nonexpansive. By applying \cite[Theorem~5.15]{bauschkecombettes}, we deduce that $(\bz^k)$ converges weakly to a point $\bz\in\Fix T$ and that $\lim_{k\to\infty}\|\bz^{k+1}-\bz^k\|=0$.

\eqref{th:main_b}:~By nonexpansivity of resolvents, $L$-Lipschitz continuity of $B_1,\dots,B_{n-1}$, and boundedness of $(\bz^k)$, it follows that $(\bx^k)$ is also bounded. Further, \eqref{eq:T_A} and the fact that $\lim_{k\to\infty}\|\bz^{k+1}-\bz^k\|=0$ implies that
\begin{equation}\label{eq:xki}
 \lim_{k\to\infty}\|x_{i}^k-x_{i-1}^k\|=0\quad\forall i=2,\dots, n.
\end{equation}

Next, using the definition of the resolvent together with \eqref{eq:th J_A}, we have
{\begin{equation}\label{eq:demiclosed2}
S\begin{pmatrix}
z_1^k-x_1^k \\
(z_2^k-x_2^k)-(z_{1}^k-x_{1}^k) +\lambda b_2^k \\
\vdots \\
(z_{n-1}^k-x_{n-1}^k)-(z_{n-2}^k-x_{n-2}^k)+\lambda b_{n-1}^k \\
x_n^k \\
\end{pmatrix}
 \ni
\begin{pmatrix}
x_1^k-x_n^k \\
x_2^k-x_n^k \\
\vdots \\
x_{n-1}^k-x_n^k\\
x_1^k-x_n^k + \lambda\displaystyle\sum_{i=1}^{n-1}b_{i+1}^k
\end{pmatrix},\small
\end{equation}}
where $b_i^k:=B_{i-1}(x_{i}^k) - B_{i-1}(x_{i-1}^k)$ and the operator $S\colon\Hilbert^n\setto\Hilbert^n$ is given by
\begin{equation}\label{eq:operator s}
 S:= \begin{pmatrix}
(\lambda A_1)^{-1}\\\bigl(\lambda (A_2+B_1)\bigr)^{-1} \\ \vdots \\ \bigl(\lambda (A_{n-1}+B_{n-2})\bigr)^{-1} \\ \lambda(A_n+B_{n-1})\\
\end{pmatrix} + \begin{pmatrix}
0 & 0 & \dots & 0 & -\Id \\
0 & 0 & \dots & 0 & -\Id \\
\vdots & \vdots & \ddots & \vdots & \vdots \\
0 & 0 & \dots & 0 & -\Id \\
\Id & \Id & \dots & \Id & 0 \\
\end{pmatrix}.
\end{equation}
As the sum of two maximally monotone operators is again maximally monotone provided that one of the operators has full domain \cite[Corollary~24.4(i)]{bauschkecombettes}, it follows that $S$ is maximally monotone. Consequently, it is demiclosed \cite[Proposition~20.38]{bauschkecombettes}. That is, its graph is sequentially closed in the weak-strong topology.

Let $\mathbf{w}\in\Hilbert^{n}$ be an arbitrary weak cluster point of the sequence $(\bx^k)$. As a consequence of \eqref{eq:xki}, $\mathbf{w}=(x,\dots,x)$ for some $x\in\Hilbert$. Taking the limit along a subsequence of $(\bx^k)$ which converges weakly to $\mathbf{w}$ in \eqref{eq:demiclosed2}, using demiclosedness of $S$ together with
$L$-Lipschitz continuity of $B_1,\dots,B_{n-1}$, and unravelling the resulting expression gives
\begin{equation*}
\left\{\begin{array}{rll}
    \lambda A_1(x) &\ni z_1-x, \\
    \lambda(A_i+B_{i-1})(x) &\ni z_i-z_{i-1} & \forall i\in\integ{2}{n-1}, \\
    \lambda(A_n+B_{n-1})(x) &\ni x-z_{n-1},
   \end{array}\right.
\end{equation*}
which implies $\bz\in\Fix T$ and $x=J_{A_1}(z_1)\in\zer\left(\sum_{i=1}^nA_i+\sum_{i=1}^{n-1}B_i\right)$.

In other words, $\mathbf{w}=(x,\dots,x)\in\Hilbert^n$ with $x:=J_{A_1}(z_1)$ is the unique weak sequential cluster point of the bounded sequence $(\bx^k)$. We therefore deduce that $(\bx^k)$ converges weakly to $\mathbf{w}$, which completes this part of the proof.

\eqref{th:main_c}:~For convenience, denote $\by^k=(y_1^k,\dots,y_n^k)$ where
\begin{equation*}\begin{cases}
    y_1^k:=z_1^k,  \\
    y_i^k:=z_i^k+x_{i-1}^k-z_{i-1}^k-\lambda B_{i-1}(x_{i-1}^k)  \quad \forall i\in\integ{2}{n-1}, \\
    y_n^k:=x_1^k+x_{n-1}^k-z_{n-1}^k-\lambda B_{n-1}(x_{n-1}^k),
   \end{cases}
\end{equation*}
so that $x_i^k=J_{\lambda A_i}(y_i^k)$ for all $i\in\integ{1}{n}$. Define $\by=(y_1,\dots,y_n)$ in an analogous way with $\bz$ in place of $\bz^k$ and $(x,\dots,x)$ in place of $\bx^k$, so that $x=J_{\lambda A_i}(y_i)$ for all $i\in\integ{1}{n}$. Using firm nonexpansivity of resolvents yields
\begin{equation}\label{eq:ci}
\begin{aligned}
0 &\leq \sum_{i=1}^{n}\langle J_{\lambda A_i}(y_i^k)-J_{\lambda A_i}(y_i), (\Id-J_{\lambda A_i})(y_i^k)-(\Id-J_{\lambda A_i})(y_i)\rangle \\
 &=\langle x_1^k-x,(z_1^k-x_1^k)-(z_1-x)\rangle \\
 &\quad     + \sum_{i=2}^{n-1}\langle x_i^k-x, (z_i^k-x_i^k)-(z_{i-1}^k-x_{i-1}^k)-\lambda B_{i-1}(x_{i-1}^k)\rangle \\
 &\quad- \sum_{i=2}^{n-1}\langle x_i^k-x,z_i-z_{i-1}-\lambda B_{i-1}(x)\rangle \\
 &\quad     + \langle x_n^k-x,x_1^k-x_n^k-(z_{n-1}^k-x_{n-1}^k)\\
 &\quad-\lambda B_{n-1}(x_{n-1}^k)\rangle-\langle x_n^k-x,x-z_{n-1}-\lambda B_{n-1}(x)\rangle\\
 &= \langle x_1^k-x_n^k,(z_1^k-x_1^k)-(z_1-x)\rangle +\langle x_n^k-x,(z_1^k-x_1^k)-(z_1-x)\rangle\\
 &\quad     + \sum_{i=2}^{n-1}\langle x_i^k-x_n^k, (z_i^k-x_i^k)-(z_{i-1}^k-x_{i-1}^k)-(z_i-z_{i-1})\rangle \\
 &\quad     + \langle x_n^k-x, (z_{n-1}^k-x_{n-1}^k)-(z_{1}^k-x_{1}^k)-(z_{n-1}-z_{1})\rangle \\
 &\quad     -\lambda\sum_{i=1}^{n-1} \langle x_{i+1}^k-x_i^k,B_{i}(x_{i}^k)-B_{i}(x)\rangle-\lambda\sum_{i=1}^{n-1} \langle x_{i}^k-x,B_{i}(x_{i}^k)-B_{i}(x)\rangle\\
 &\quad     + \langle x_n^k-x,x_1^k-x_n^k\rangle-\langle x_n^k-x,(z_{n-1}^k-x_{n-1}^k)+(x-z_{n-1})\rangle.
\end{aligned}
\end{equation}
Rearranging \eqref{eq:ci} followed by applying $\frac{1}{L}$-cocoercivity of $B_1,\dots,B_{n-1}$ gives
\begin{multline}\label{eq:cii}
\langle x^k_n-x,x_1^k-x_n^k\rangle
+ \langle x^k_1-x_n^k,(z^k_1-x^k_1)-(z_1-x)\rangle \\ -\lambda \sum_{i=1}^{n-1}\langle x^k_{i+1}-x^k_{i},B_i(x_{i}^k)-B_i(x)\rangle \\
  + \sum_{i=2}^{n-1}\langle x^k_i-x_n^k,((z_i^k-x_i^k)-(z_{i-1}^k-x_{i-1}^k))-(z_i-z_{i-1})\rangle \\
  \geq \lambda\sum_{i=1}^{n-1}\langle x^k_{i}-x,B_i(x_{i}^k)-B_i(x)\rangle \geq  \frac{\lambda}{L}\sum_{i=1}^{n-1}\|B_i(x_{i}^k)-B_i(x)\|^2.
\end{multline}
Note that the left-hand side of \eqref{eq:cii} converges to zero due to \eqref{eq:xki} and the boundedness of sequences $(\bz^k),(\bx^k)$ and $(B_i(x_{i}^k))$ for $i\in\integ{1}{n-1}$. It then follows that $B_i(x^k_{i})\to B_i(x)$ for all $i\in\integ{1}{n-1}$, as claimed.
\end{proof}

\begin{remark}[Attouch--Th\'era duality]
Let $I\subseteq\{1,\dots,n-1\}$ be a non-empty index set with  cardinality denoted by $\lvert I\rvert$. Express the monotone inclusion \eqref{eq:mono inc} as
\begin{equation}\label{eq:primal inclusion}
 \text{find~}x\in\Hilbert\text{~such that~}0\in\sum_{i\in I}B_i(x)+\left(\sum_{i=1}^nA_i+\sum_{i\not\in I}B_i\right)(x),
\end{equation}
and note that the first operator $\sum_{i\in I}B_i$ is $\frac{1}{\lvert I\rvert L}$-cocoercive (see, e.g.,~\cite[Proposition~4.12]{bauschkecombettes}). The \emph{Attouch--Th{\'e}ra dual}~\cite{attouch1996general} associated with \eqref{eq:primal inclusion} takes the form
\begin{equation}\label{eq:dual inclusion}
 \text{find~}u\in\Hilbert\text{~such that~}0\in \left(\sum_{i\in I}B_i\right)^{-1}(u)-\left(\sum_{i=1}^nA_i+\sum_{i\not\in I}B_i\right)^{-1}(-u),
\end{equation}
where we note that the first operator $\left(\sum_{i\in I}B_i\right)^{-1}$ is $\frac{1}{\lvert I\rvert L}$-strongly monotone. Hence, as a strongly monotone inclusion, \eqref{eq:dual inclusion} has a unique solution $\bar{u}\in\Hilbert$. Moreover, for any solution $\bar{x}\in\Hilbert$ of \eqref{eq:primal inclusion},  \cite[Theorem~3.1]{attouch1996general} implies $\bar{u}=\left(\sum_{i\in I}B_i\right)(\bar{x})$. In the context of the previous result, Theorem~\ref{th:main}\eqref{th:main_c} implies $\sum_{i\in I}B_i(x^k_i)\to\bar{u}$ as $k\to\infty$. In other words, the algorithm in~\eqref{eq:alg1} also produces a sequence which converges strongly to the unique solution of the dual inclusion \eqref{eq:dual inclusion}.
\end{remark}

\begin{remark}
(i) When $B_1 =\dots=B_{n-1}=0$, Theorem~\ref{th:main} recovers \cite[Theorem~4.5]{malitsky2021resolvent}.\\
(ii) In the special case when $n=2$, \eqref{eq:T_A ne} from Lemma~\ref{l:T_A ne} simplifies to give the stronger inequality
\begin{equation}\label{eq:ne 2}
\|T(\mathbf{z})-T(\mathbf{\bar{z}})\|^2 +\left(\frac{2-\gamma}{\gamma}- \frac{\lambda L}{2\gamma}\right)\|(\Id-T)(\mathbf{z})-(\Id-T)(\bar{\mathbf{z}})\|^2\leq \|\mathbf{z} - \mathbf{\bar{z}}\|^2.
\end{equation}
This assures averagedness of $T$ provided that $\gamma\in\bigl(0,2-\frac{\lambda L}{2}\bigr)$, which is larger than the range of permissible values for $\gamma$ in the statement of Theorem~\ref{th:main}. However, by using \eqref{eq:ne 2}, a proof similar to that of Theorem~\ref{th:main} guarantees the convergence for a larger range of parameter values, namely, when $\lambda \in{\bigl( 0,\frac{4}{L}\bigr)}$ and $\gamma\in{\bigl(0,2-\frac{\lambda L}{2}\bigr)}$. For details, see \cite{dao2021adaptive,aragon2021direct}.
\end{remark}

\section{A Distributed Forward-Reflected-Backward Method}\label{s:dforb}
Let $n\geq 3$ and consider the problem
\begin{equation}\label{eq:mono inc lip}
 \text{find~}x\in\Hilbert\text{~such that~}0\in\left(\sum_{i=1}^nA_i+\sum_{i=1}^{n-2}B_i\right)(x),
\end{equation}
where $A_1,\dots,A_n\colon\Hilbert\setto\Hilbert$ are maximally monotone and $B_1,\dots,B_{n-2}\colon\Hilbert\to\Hilbert$ are monotone and $L$-Lipschitz continuous.

Developing splitting algorithms which use forward evaluations of Lipschitz
continuous monotone operators is generally more intricate than those exploiting
cocoercivity, such as the one in the previous section. For concreteness, consider the special case of \eqref{eq:mono inc lip} with two operators given by
\begin{equation}\label{eq:two op}
 \text{find~}x\in\Hilbert\text{~such that~}0\in\left(A_1+B_1\right)(x).
\end{equation}
It is well known that the \emph{forward-backward method} for \eqref{eq:two op} given by
\begin{equation}\label{eq:fb}
x^{k+1} = J_{\la A_1}(x^k-\la B_1(x^k))
\end{equation}
fails to converge for any $\la >0$. Indeed, consider the particular instance of~\eqref{eq:two op} given by  $\Hilbert = \mathbb{R}^2$, $A_1:=0$ and $B_1:=\left(\begin{smallmatrix} 0 & -1 \\ 1 & 0 \end{smallmatrix}\right)$, whose unique solution is $(0,0)^T$. Then, $B_1$ is skew-symmetric and thus monotone (but not cocoercive), but the sequence generated by~\eqref{eq:fb} will diverge for any non-zero starting point, since the eigenvalues of $\Id-\lambda B_1$ are $1\pm \lambda i$. However, a small modification of \eqref{eq:fb} gives rise to
\begin{equation}
  \label{eq:forb}
  x^{k+1} = J_{\la A_1}\bigl(x^k -2\la B_1(x^k) + \la B_1(x^{k-1})\bigr),
\end{equation}
which is known as the \emph{forward-reflected-backward
  method}~\cite{malitsky2018forward}. Unlike \eqref{eq:fb}, it  converges for any $\la < \frac{1}{2L}$. While \eqref{eq:forb} is not the only constant stepsize scheme for solving \eqref{eq:two op}, as there are a few which are fundamentally different~\cite{tseng2000modified,csetnek2019shadow}, it is arguably one of the simplest. In this section, we develop a modification of the method from the previous section which converges for Lipschitz continuous operators by drawing inspiration from the differences between \eqref{eq:forb} and \eqref{eq:fb}.
\smallskip

\begin{mdframed}
Given $\lambda\in\bigl(0,\frac{1}{2L}\bigr)$ and $\gamma\in\bigl(0,1-2\lambda L\bigr)$ and an initial point $\bz^0=(z_1^0,\dots,z_{n-1}^0)\in\Hilbert^{n-1}$, our proposed algorithm for~\eqref{eq:mono inc lip} generates two sequences, $(\bz^k)\subseteq\Hilbert^{n-1}$ and $(\bx^k)\subseteq\Hilbert^n$, according to
\begin{subequations}\label{eq:alg2}
\begin{equation}\small\label{eq:th T_A lip}
\bz^{k+1}
=  \bz^k +
\gamma\begin{pmatrix}
x_2^k-x_1^k \\
x_3^k-x_2^k \\
\vdots \\
x_{n}^k-x_{n-1}^k \\
\end{pmatrix}
\end{equation}
and
\begin{equation}\label{eq:th J_A lip}
\footnotesize\left\{\begin{aligned}
    x_1^k&=J_{\lambda A_1}\bigl( z_1^k\bigr),  \\
    x_2^k&=J_{\lambda A_2}\bigl( z_2^k+x_{1}^k-z_{1}^k-\lambda B_{1}(x_{1}^k) \bigr) , \\
    x_i^k&=J_{\lambda A_i}\bigl( z_i^k+x_{i-1}^k-z_{i-1}^k-\lambda B_{i-1}(x_{i-1}^k)-\lambda (B_{i-2}(x_{i-1}^k) - B_{i-2}(x_{i-2}^k)) \bigr), \\
    x_n^k&=J_{\lambda A_n}\bigl( x_1^k+x_{n-1}^k-z_{n-1}^k-\lambda (B_{n-2}(x_{n-1}^k) - B_{n-2}(x_{n-2}^k)) \bigr).
\end{aligned}\right.\small
 \end{equation}
 for  $i\in\integ{3}{n-1}.$
\end{subequations}
\end{mdframed}

Compared to the algorithm proposed in the previous section, the only major change here is that some expressions for $x^k_i$ in \eqref{eq:th J_A lip} incorporate a ``reflection-type'' term involving the operator $B_{i-2}$. This precise form seems important for our subsequence convergence analysis and it seems not easy to incorporate ``reflection-type'' terms involving the operator $B_{i-1}$. The structure of \eqref{eq:alg2} allows for a similar protocol to the one described in Algorithm~\ref{alg:DFB proto} to be used for a distributed decentralised implementation. The only change to the protocol (in terms of communication) is that Agent $i$ must also now send $\lambda\bigl(B_{i-1}(x_{i}^k)-B_{i-1}(x_{i-1}^k)\bigr)$ to Agent $i+1$ for all $i\in\llbracket 2,n-1\rrbracket$.

\begin{remark}
To the best of our knowledge, the scheme given by \eqref{eq:alg2} does not directly recover any existing forward-backward-type scheme as special case (although it is clearly related to \eqref{eq:forb}). For example, if we take $n=3$ and $A_1=A_3=0$. Then $x_1^k$ and $x_3^k$ can be eliminated from~\eqref{eq:alg2} to give
\begin{equation*}
\left\{\begin{aligned}
  x_2^k &= J_{\lambda A_2}\big( z_2^k-\lambda B_1(z_1^k) \bigr) \\
  z^{k+1}_1 &= z^k_1 + \gamma\bigl(x_2^k-z_1^k\bigr) \\
  z^{k+1}_2 &= z^k_2 + \gamma\bigl(z_1^k-z_2^k-\lambda(B_1(x_2^k)-B_1(z_1^k))\bigr).
 \end{aligned}\right.
\end{equation*}
To better understand the relationship between this and \eqref{eq:forb}, it is instructive to consider the limiting case with $\gamma=1$. Indeed, when $\gamma=1$, $x_2^{k}$ and $z_2^k$ can be eliminated to give
$$ z^{k+1}_1 = J_{\lambda A_2}\big( z_1^{k-1}-2\lambda B_1(z_1^{k})+\lambda B_1(z_1^{k-1})\bigr). $$
Although this closely resembles \eqref{eq:forb} for finding zero of $A_2+B_1$, it is not exactly the same due to the index of the first term inside the resolvent.
\end{remark}

In order to analyse \eqref{eq:alg2}, we introduce the underlying fixed point operator $\widetilde{T}\colon\Hilbert^{n-1}\to\Hilbert^{n-1}$ given by
\begin{equation}\label{eq:T lip}
\widetilde{T}(\bz)
:=  \bz +
\gamma\begin{pmatrix}
x_2-x_1 \\
x_3-x_2 \\
\vdots \\
x_{n}-x_{n-1} \\
\end{pmatrix}
\end{equation}
where $\bx=(x_1,\dots,x_n)\in\Hilbert^n$ depends on $\bz=(z_1,\dots,z_n)\in\Hilbert$ and is given by
\begin{equation}{\footnotesize\label{eq:xi lip}\left\{\begin{aligned}
    x_1&=J_{\lambda A_1}\bigl( z_1\bigr),  \\
    x_2&=J_{\lambda A_2}\bigl( z_2+x_{1}-z_{1}-\lambda B_{1}(x_{1}) \bigr) , \\
    x_i&=J_{\lambda A_i}\bigl( z_i+x_{i-1}-z_{i-1}-\lambda B_{i-1}(x_{i-1})-\lambda (B_{i-2}(x_{i-1}) - B_{i-2}(x_{i-2})) \bigr), \\
    x_n&=J_{\lambda A_n}\bigl( x_1+x_{n-1}-z_{n-1}-\lambda (B_{n-2}(x_{n-1}) - B_{n-2}(x_{n-2})) \bigr),
   \end{aligned}\right.}
 \end{equation}%
for $i\in\integ{3}{n-1}$.
In this way, the sequence $(\bz^k)$ given by \eqref{eq:alg2}  satisfies $\bz^{k+1}=\widetilde{T}(\bz^k)$ for all $k\in\mathbb{N}$.

Next, we analyse the nonexpansivity properties of the operator~$\widetilde{T}$. The proof of the following result is similar to that of Lemma~\ref{l:T_A ne}, but using the Lipschitzian properties of the operators $B_1,\ldots,B_{n-2}$ instead of cocoercivity.
\begin{lemma}\label{l:T_A ne_lip}
Let $\bar{\bz}=(\bar{z}_1,\dots,\bar{z}_{n-1})\in\Fix\widetilde{T}$. Then, for all $\bz=(z_1,\dots,z_{n-1})\in\Hilbert^{n-1}$, we have
\begin{multline}\label{eq:T_A ne lip}
 \|\widetilde{T}(\bz)-\bar{\bz}\|^2 + \left(\frac{1-\gamma}{\gamma}-\frac{2\lambda L}{\gamma}\right)\|(\Id-\widetilde{T})(\bz)\|^2
 + \frac{1}{\gamma}\bigl\|\sum_{i=1}^{n-1}(\Id-\widetilde{T})(\bz)_i\bigr\|^2 \\
 + \gamma \lambda L \|(\Id-\widetilde{T})(\bz)_1\|^2 + \gamma \lambda L\|(\Id-\widetilde{T})(\bz)_{n-1}\|^2\leq
\|\bz-\bar{\bz} \|^2.
\end{multline}
In particular, if $\lambda\in(0,\frac{1}{2L})$ and $\gamma\in(0,1-2\lambda L)$, then $\widetilde{T}$ is $\sigma$-strongly quasi-nonexpansive for $\sigma=\frac{1-\gamma}{\gamma}-\frac{2\lambda L}{\gamma}>0$.
\end{lemma}
\begin{proof}
For convenience, denote $\bz^+=\widetilde{T}(\bz)$. Further, let $\bx=(x_1,\dots,x_n)\in\Hilbert^n$ be given by \eqref{eq:xi lip} and let $\bar{\bx}=(\bar{x},\dots,\bar{x})\in\Hilbert^{n-1}$ be given analogously. Note that the expression of $\bar{\bx}$ is justified as $\bar{\bz}=\widetilde{T}(\bar{\bz})$. Monotonicity of $\lambda A_1$ implies
\begin{equation}\label{eq:qneA1}
0  \leq \langle x_2-\bar{x}, (z_1-x_1)-(\bar{z}_1-\bar{x})\rangle  + \langle x_1-x_2, (z_1-x_1)-(\bar{z}_1-\bar{x})\rangle .
\end{equation}
In order to simplify the case study, we introduce the zero operator $B_0:=0$. By monotonicity of $\lambda A_i$, we deduce
\begin{equation}\label{eq:qneAi}
\begin{aligned}
 0 &\leq      \langle x_{i+1}-\bar{x},(z_i-x_i) - (\bar{z}_i - \bar{x})\rangle + \langle x_i-x_{i+1},  (z_{i}-x_{i})-(\bar{z}_{i}-\bar{x})\rangle            \\
 & \quad - \langle x_{i}-\bar{x},  (z_{i-1}-x_{i-1})-(\bar{z}_{i-1}-\bar{x})\rangle   \\
 &\quad-\lambda\langle x_i-\bar{x},B_{i-1}(x_{i-1})-B_{i-1}(\bar{x})\rangle \\
 & \quad -   \lambda \langle x_i-\bar{x},B_{i-2}(x_{i-1})-B_{i-2}(x_{i-1})\rangle,
\end{aligned}
\end{equation}
and monotonicity of $\lambda A_n$ yields
\begin{equation}\label{eq:qneAn}
\begin{aligned}
0 &\leq   -\langle x_n-\bar{x}, (z_{n-1}-x_{n-1}) - (\bar{z}_{n-1}-\bar{x})\rangle\\
&\quad  - \lambda \langle x_n-\bar{x}, B_{n-2}(x_{n-1})-B_{n-2}(x_{n-2})\rangle \\
&\quad  + \frac{1}{2}\left( \|x_1-\bar{x}\|^2 - \|x_n-\bar{x}\|^2 + \|x_1-x_n\|^2\right).
\end{aligned}
\end{equation}
Summing together \eqref{eq:qneA1}-\eqref{eq:qneAn}, we obtain the inequality
\begin{equation}\label{eq:qnesum}
\begin{aligned}
0 & \leq \sum_{i=1}^{n-1} \langle (\bar{z}_i-\bar{x})-(z_i-x_i),x_{i+1}-x_i\rangle \\
&\quad+ \frac{1}{2}\left( \|x_1-\bar{x}\|^2 - \|x_n-\bar{x}\|^2 + \|x_1-x_n\|^2\right) \\
& \quad - \lambda \sum_{i=2}^{n-1} \langle x_i-\bar{x},B_{i-1}(x_{i-1})-B_{i-1}(\bar{x}) \rangle \\
& \quad  - \lambda\sum_{i=3}^n\langle x_i-\bar{x},B_{i-2}(x_{i-1})-B_{i-2}(x_{i-2})\rangle,
\end{aligned}
\end{equation}
where we have omitted the index $i=2$ in the last sum, since $B_0:=0$. The first term in~\eqref{eq:qnesum} multiplied by $2\gamma$ can be written as
\begin{equation}\label{eq:qnesum1}
\begin{aligned}
 2\gamma \sum_{i=1}^{n-1} &\langle (\bar{z}_i-\bar{x})-(z_i-x_i),x_{i+1}-x_i\rangle  \\
& = \sum_{i=1}^{n-1} \left(\|\bar{z}_i-z_i\|^2 + \|z_i^+-z_i\|^2 - \| z_i^+-\bar{z}_i\|^2 \right) \\
& \quad-\frac{1}{\gamma} \sum_{i=1}^{n-1}  \|z_i^+-z_i\|^2 + \gamma \left( \|x_{n}-\bar{x}\|^2 - \|x_1-\bar{x}\|^2\right).
\end{aligned}
\end{equation}
Therefore, multiplying~\eqref{eq:qnesum} by $2\gamma$ and substituting~\eqref{eq:qnesum1}, we reach the inequality                                 
\begin{multline}\label{eq:key lip}
 \|\widetilde{T}(\bz)-\bar{\bz}\|^2 + \frac{1-\gamma}{\gamma}\|(\Id-\widetilde{T})(\bz)\|^2
 + \frac{1}{\gamma}\bigl\|\sum_{i=1}^{n-1}(\Id-\widetilde{T})(\bz)_i\bigr\|^2 \\
 \leq
\|\bz-\bbz \|^2-2\gamma\lambda\sum_{i=2}^{n-1}\langle x_i-\bar{x},B_{i-1}(x_{i-1})-B_{i-1}(\bar{x})\rangle\\
-2\gamma\lambda\sum_{i=3}^{n}\langle x_i-\bar{x}, B_{i-2}(x_{i-1}) - B_{i-2}(x_{i-2})\rangle.
\end{multline}
Using monotonicity of $B_1,\dots,B_{n-2}$, the second last term can be estimated as
\begin{equation}\label{eq:key lip 2}
-\sum_{i=2}^{n-1}\langle x_i-\bar{x},B_{i-1}(x_{i-1})-B_{i-1}(\bar{x})\rangle
 \leq \sum_{i=2}^{n-1}\langle x_i-\bar{x},B_{i-1}(x_i)-B_{i-1}(x_{i-1})\rangle
\end{equation}
and, using $L$-Lipschitz continuity of $B_1,\dots,B_{n-2}$, the last term can be estimated as
\begin{equation}\label{eq:key lip 3}
\begin{aligned}
-\sum_{i=3}^{n}&\langle x_i-\bar{x}, B_{i-2}(x_{i-1}) - B_{i-2}(x_{i-2})\rangle \\
&=-\sum_{i=3}^{n}\langle x_{i-1}-\bar{x}, B_{i-2}(x_{i-1}) - B_{i-2}(x_{i-2})\rangle\\
&\quad+ \sum_{i=3}^{n}\langle x_{i-1}-x_i, B_{i-2}(x_{i-1}) - B_{i-2}(x_{i-2})\rangle\\
&\leq -\sum_{i=3}^{n}\langle x_{i-1}-\bar{x}, B_{i-2}(x_{i-1}) - B_{i-2}(x_{i-2})\rangle\\
&\quad+ \frac{L}{2}\sum_{i=3}^{n}\left(\|x_{i-1}-x_i\|^2+\|x_{i-1} - x_{i-2}\|^2\right)\\
&= -\sum_{i=2}^{n-1}\langle x_{i}-\bar{x}, B_{i-1}(x_{i}) - B_{i-1}(x_{i-1})\rangle
+L\sum_{i=2}^{n}\|x_i-x_{i-1}\|^2 \\
&\qquad- \frac{L}{2}\|x_2-x_1\|^2 - \frac{L}{2}\|x_n-x_{n-1}\|^2 \\
&= -\sum_{i=2}^{n-1}\langle x_{i}-\bar{x}, B_{i-1}(x_{i}) - B_{i-1}(x_{i-1})\rangle +\frac{L}{\gamma^2}\|(\Id-\widetilde{T})(\bz)\|^2\\
&\qquad
 - \frac{L}{2}\|(\Id-\widetilde{T})(\bz)_1\|^2 - \frac{L}{2}\|(\Id-\widetilde{T})(\bz)_{n-1}\|^2.
\end{aligned}
\end{equation}
Thus, substituting \eqref{eq:key lip 2} and \eqref{eq:key lip 3} into \eqref{eq:key lip} gives \eqref{eq:T_A ne lip}, which completes the proof.
\end{proof}

\begin{remark}
Compared to Lemma~\ref{l:T_A ne} from the previous section, the conclusions of Lemma~\ref{l:T_A ne_lip} are weaker in two ways. Firstly, the permissible stepsize range of $\lambda\in (0,\frac{1}{2L})$ is smaller than in Lemma~\ref{l:T_A ne}, which allowed $\lambda\in(0,\frac{2}{L})$. And, secondly, the operator $\widetilde{T}$ is only shown to be strongly quasi-nonexpansive in Lemma~\ref{l:T_A ne_lip} whereas $T$ is known to be averaged nonexpansive.
\end{remark}

The following theorem is our main result regarding convergence of \eqref{eq:alg2}.

\begin{theorem}\label{th:main 2}
Let $n\geq 3$, let $A_1,\dots,A_n\colon\Hilbert\setto\Hilbert$ be maximally monotone and let $B_1,\dots,B_{n-2}\colon\Hilbert\to\Hilbert$ be monotone and $L$-Lipschitz continuous with $\zer\left(\sum_{i=1}^nA_i+\sum_{i=1}^{n-2}B_i\right)\neq\varnothing$. Further, let $\lambda\in\bigl(0,\frac{1}{2L}\bigr)$ and $\gamma\in\bigl(0,1-2\lambda L\bigr)$. Given $\bz^0\in\Hilbert^{n-1}$, let $(\bz^k)\subseteq\Hilbert^{n-1}$ and $(\bx^k)\subseteq\Hilbert^n$ be the sequences given by \eqref{eq:alg2}.
Then the following assertions hold.
\begin{enumerate}[(a)]
\item\label{th:main2_a} The sequence $(\bz^k)$ converges weakly to a point $\bz\in\Fix \widetilde{T}$.
\item\label{th:main2_b} The sequence $(\bx^k)$ converges weakly to a point $(x,\dots,x)\in\Hilbert^n$ with $x\in\zer\left(\sum_{i=1}^nA_i+\sum_{i=1}^{n-2}B_i\right)$.
\end{enumerate}
\end{theorem}
\begin{proof}
\eqref{th:main2_a}:~Since $\zer\left(\sum_{i=1}^nA_i+\sum_{i=1}^{n-2}B_i\right)\neq\varnothing$, Lemma~\ref{l:fixed points}\eqref{l:fixed points b} implies that the set of fixed points of operator $T$ in~\eqref{eq:T_A}-\eqref{eq:xi} (with $B_{n-1}=0$) is nonempty. The latter set coincides with the set of fixed points of operator $\widetilde{T}$ in~\eqref{eq:T lip}-\eqref{eq:xi lip}, so $\Fix \widetilde{T}\neq\varnothing$. Since $\lambda\in\bigl(0,\frac{1}{2L})$ and $\gamma\in\bigl(0,1-2\lambda L\bigr)$, Lemma~\ref{l:T_A ne_lip} implies that $(\bz^k)$ is Fej\'er monotone with respect to $\Fix \widetilde{T}$ and that $\lim_{k\to+\infty}\|\bz^{k+1}-\bz^k\|=0$. By nonexpansivity of resolvents, $L$-Lipschitz continuity of $B_2,\dots,B_{n-1}$, and boundedness of $(\bz^k)$, it follows that $(\bx^k)$ is also bounded. Further, \eqref{eq:T lip} and the fact that $\lim_{k\to\infty}\|\bz^{k+1}-\bz^k\|=0$ implies that
\begin{equation}\label{eq:xki lips}
 \lim_{k\to\infty}\|x_{i}^k-x_{i-1}^k\|=0\quad\forall i=2,\dots, n.
\end{equation}
Let $\bu=(u_1,\dots,u_{n-1})\in\Hilbert^{n-1}$ be an arbitrary weak cluster point of $(\bz^k)$. Then, due to \eqref{eq:xki lips}, there exists a point $x\in\Hilbert$ such that $(\bu,\bw)$ is a weak cluster point of $(\bz^k,\bx^k)$, where $\bw=(x,\dots,x)\in\Hilbert^n$. Let $S$ denote the maximally monotone operator defined by \eqref{eq:operator s} when $B_{n-1}=0$.
Then \eqref{eq:th J_A lip} implies
\begin{multline}\label{eq:demiclosed 2}
S\begin{pmatrix}
z_1^k-x_1^k \\
(z_2^k-x_2^k)-(z_{1}^k-x_{1}^k) +\lambda b_2^k \\
(z_3^k-x_3^k)-(z_2^k-x_2^k) + \lambda b_3^k - \lambda b_2^k \\
\\
\vdots \\
(z_{n-1}^k-x_{n-1}^k)-(z_{n-2}^k-x_{n-2}^k)+\lambda b_{n-1}^k - \lambda b_{n-2}^k \\
x_n^k \\
\end{pmatrix}
 \ni
\begin{pmatrix}
x_1^k-x_n^k \\
x_2^k-x_n^k \\
x_3^k-x_n^k \\
\vdots \\
x_{n-1}^k-x_n^k\\
x_1^k-x_n^k\\
\end{pmatrix},
\end{multline}%
where $b_i^k:=B_{i-1}(x_{i}^k) - B_{i-1}(x_{i-1}^k)$.
Taking the limit along a subsequence of $(\bz^k,\bx^k)$ which converges weakly to $(\bu,\bw)$ in \eqref{eq:demiclosed 2}, using demiclosedness of $S$ together with $L$-Lipschitz continuity of $B_2,\dots,B_{n-1}$, and  unravelling the resulting expression gives that $\bu\in\Fix \widetilde{T}$ and $x=J_{\lambda A_1}(u_1)\in\zer\left(\sum_{i=1}^nA_i+\sum_{i=1}^{n-2}B_i\right)$. Thus, by \cite[Theorem~5.5]{bauschkecombettes}, it follows that $(\bz^k)$ converges weakly to a point $\bz\in \Fix \widetilde{T}$.

\eqref{th:main2_b}:~Follows by using an argument analogous to the one in Theorem~\ref{th:main}\eqref{th:main_b}.
\end{proof}

\begin{remark}[Exploiting cocoercivity]
If a Lipschitz continuous operator $B_i$ in \eqref{eq:mono inc lip} is actually cocoercive, then it is possible to reduce the number evaluations of $B_i$ per iteration by combining the ideas in Sections~\ref{s:dfb} and \ref{s:dforb}. In fact, we can consider the problem
\begin{equation*}
 \text{find~}x\in\Hilbert\text{~such that~}0\in\left(\sum_{i=1}^nA_i+\sum_{i=1}^{n-1}B_i\right)(x),
\end{equation*}
where $B_1,\dots,B_{n-2}$ are each either monotone and Lipschitz continuous or cocoercive, and $B_{n-1}$ is cocoercive. For this problem, we can replace \eqref{eq:xi lip} in the definition of $\widetilde{T}$ with
\begin{equation*}\left\{\begin{aligned}
    x_1&=J_{\lambda A_1}\bigl( z_1\bigr),  \\
    x_2&=J_{\lambda A_2}\bigl( z_2+x_{1}-z_{1}-\lambda B_{1}(x_{1}) \bigr) , \\
    x_i&=J_{\lambda A_i}\bigl( z_i+x_{i-1}-z_{i-1}-\lambda B_{i-1}(x_{i-1})-\lambda b_{i-1} \bigr)  \quad \forall i\in\integ{3}{n-1}, \\
    x_n&=J_{\lambda A_n}\bigl( x_1+x_{n-1}-z_{n-1}-\lambda B_{n-1}(x_{n-1})-\lambda b_{n-1} \bigr),
   \end{aligned}\right.
 \end{equation*}%
where $b_2,\dots,b_{n-1}\in\Hilbert$ are given by
 $$ b_i = \begin{cases}
   0 &\text{if $B_{i-1}$ is cocoercive}, \\
   B_{i-1}(x_{i}) - B_{i-1}(x_{i-1})  & \text{if $B_{i-1}$ is monotone and Lipschitz}.
 \end{cases}$$
This modification can be shown to converge using a proof similar to Theorem~\ref{th:main 2} for $\lambda\in(0,\frac{1}{2L})$. However, it is not straightforward to recover Theorem~\ref{th:main} as a special case of such a result because the stepsizes range of $\lambda\in(0,\frac{2}{L})$ in the cocoercive only case (i.e., Theorem~\ref{th:main}) are larger than the range in the mixed case. Moreover, Theorem~\ref{th:main}\eqref{th:main_c} (strong convergence to dual solutions) does not have an analogue in the statement of Theorem~\ref{th:main 2}. In addition, keeping the two cases separate allows the analysis to be as transparent as possible.
\end{remark}
\small

\paragraph{Acknowledgments}
FJAA and DTB were partially supported by the Ministry of Science, Innovation and Universities of Spain and the European Regional Development Fund (ERDF) of the European Commission, grant PGC2018-097960-B-C22. FJAA was partially supported by the Generalitat Valenciana (AICO/2021/165).
YM was supported by the Wallenberg Al, Autonomous Systems and Software Program (WASP) funded by the Knut
and Alice Wallenberg Foundation.  The project number is 305286.
MKT was supported in part by Australian Research Council grant DE200100063.
DTB was supported by MINECO and European Social Fund (PRE2019-090751) under the program ``Ayudas para contratos predoctorales para la formaci\'{o}n de doctores''
2019.

FJAA will always be indebted to his mentor, Professor Asen L. Dontchev, for valuable career (and life) advice as well as for sparking his interest in the analysis of optimisation algorithms under Lipschitzian properties.
The authors are thankful to the anonymous referees for their careful reading and for providing very helpful comments.

\end{document}